\documentclass{article}

\usepackage{arxiv}

\usepackage[utf8]{inputenc} 
\usepackage[T1]{fontenc}    
\usepackage{hyperref}       
\usepackage{url}            
\usepackage{booktabs}       
\usepackage{amsfonts}       
\usepackage{nicefrac}       
\usepackage{microtype}      
\usepackage{lipsum}	
\usepackage{amsmath}
\usepackage{amsthm}
\newtheorem{definition}{Definition}
\newtheorem{remark}{Remark}
\newtheorem{theorem}{Theorem}
\newtheorem{lemma}{Lemma}

\title{Existence results for doubly nonlinear parabolic equations with two lower order terms and $L^1$-data}

\author{
 Abdelmoujib Benkirane\\
   Department of Mathematics\\
   Faculty of Sciences Fez\\
  University Sidi Mohamed Ben Abdellah \\
  P. O. Box 1796 Atlas Fez, Morocco.\\
  \texttt{abd.benkirane@gmail.com} \\
   \And
Youssef El Hadfi \\
Laboratory LIPIM\\
 National School of Applied Sciences\\
  Sultan M. Slimane University \\
  P. O. Box 77 Khouribga Morocco.\\
  \texttt{yelhadfi@gmail.com} \\
     \And
Mostafa El Moumni  \\
 Department of Mathematics\\
 Faculty of Sciences El Jadida\\
   University Chouaib Doukkali \\
 P. O. Box 20. 24000 El Jadida, Morocco.\\
  \texttt{mostafaelmoumni@gmail.com} \\
}

\begin{document}
\maketitle
\begin{abstract}
We investigate the existence of a renormalized solution for a class of nonlinear parabolic equations with two lower order terms and $L^1$-data.
\end{abstract}

\keywords{Nonlinear parabolic equations \and Renormalized solutions}

\section{Introduction}\quad We consider the following nonlinear parabolic problem
\begin{equation}\left\{ \begin{array}{lll}
\displaystyle{\partial b(x,u)\over \partial t}-{\rm div}(a(x,t,u,\nabla u))+ g(x,t,u,\nabla u)+ H(x,t,\nabla u) = f  &\mbox{ in } & Q_T, \\
b(x,u)(t=0)=b(x,u_0)  & \mbox{ in } &\Omega,\\
u  =  0 &\mbox{ on }&  \partial \Omega\times (0,T),
\end{array} \right. \label{D1.1}\end{equation}
where $\Omega$ is a bounded open subset of $\mathbb{R}^N$, $N \geq 1,\ T > 0,\ p > 1$ and $Q_T$ is the cylinder $\Omega\times(0, T)$.  The operator  $-{\rm div}(a(x,t,u,\nabla u))$ is a Leray-Lions operator which is coercive and grows like $|\nabla u|^{p-1}$ with respect to $\nabla u$, the function $b(x,u)$ is an unbounded on $u$, and $b(x,u_0) \in L^1(\Omega)$. The functions $g$ and $H$ are two Carath\'{e}odory functions with suitable assumptions see below. Finally the datum $f\in L^1(Q_T)$.
\par The problem \eqref{D1.1} is encountered in a variety of physical phenomena and applications. For instance, when $b(x,u)=u,\,a(x,t,u,\nabla u)=|\nabla u|^{p-2}\nabla u,\, g = f= 0, H(x,t,\nabla u)= \lambda |\nabla u|^q,$ where $q$ and $\lambda$ are positive parameters, the equation in problem \eqref{D1.1} can be viewed as the viscosity approximation of Hamilton -Jacobi type equation from stochastic control theory \cite{L2011}. In particular, when $b(x,u)=u,\,a(x,t,u,\nabla u)=\nabla u,\, g = f= 0, H(x,t,\nabla u)= \lambda |\nabla u|^2,$ where $\lambda$ is positive parameters, the equation in problem \eqref{D1.1} appears in the physical theory of growth and roughening of surfaces, where it is known as the Kardar-Parisi-Zhang equation \cite{KPZ1986}. We introduce the definition of the renormalized solutions for problem \eqref{D1.1} as follows. This notion was introduced by P.-L. Lions and Di Perna \cite{D.L} for the study of Boltzmann equation (see also P.-L. Lions \cite{Lions96} for a few applications to fluid mechanics models). This notion was then adapted to an elliptic version of \eqref{D1.1} by Boccardo et al \cite{B.G.D.M} when the right hand side is in $W^{-1,p'}(\Omega)$, by  Rakotoson \cite{Ra} when the right hand side being a in $L^1(\Omega)$, and by  Dal Maso,  Murat, Orsina and  Prignet \cite{D.M.O.P}  for the case of right hand side being a general measure data, see also \cite{Murat1,Murat2}.
\par For $b(x,u)=u$ and $H=0$, the existence of a weak solution to Problem \eqref{D1.1} (which belongs to $L^m(0, T;W^{1,m}_0(\Omega))$ with $p > 2 - \frac{1}{N+1}$ and $m < \frac{p(N+1)-N}{N+1}$ was proved in \cite{Boc1} (see also \cite{BoDaGaOr97}) where $g=0$, and in \cite{Po} where $g= 0$, and in \cite{Dall1,Po1,Po2}. When the function $g(x, t, u,\nabla u)\equiv g(u)$ is independent on the $(x, t,\nabla u)$ and $g$ is continuous, the existence of a renormalized solution to problem \eqref{D1.1} is proved in \cite{B.M.R}. Otherwise, recently in \cite{A.B.E.R2014} is proved the existence of a renormalized solution to problem \eqref{D1.1} where the variational case.
\par The scope of the present paper is to prove an existence result for renormalized solutions to a class of problems \eqref{D1.1} with two lower order terms and $L^1$-data. The difficulties connected to our problem \eqref{D1.1} are due to the presence of the two terms $g$ and $H$ which induce a lack of coercivity, noncontrolled growth of the function $b(x,s)$ with respect to $s$, the functions $a(x, t, u,\nabla u)$ do not belong to $(L^1_{loc}(Q_T))^N$ in general, and the data $b(x,u_0),\ f$ are only integrable.
\par The rest of this article is organized as follows. In Section 2 we make precise all the assumptions on $b,\ a,\ g,\ H,\ u_0$, we also give the concept of a renormalized solution for the problem \eqref{D1.1}.  In section 3 we establish the existence of our main results.
\section{Essential assumptions and different notions of solutions}
\quad Throughout the paper, we assume that the following assumptions hold true.  Let $\Omega$ is a bounded open set of $\mathbb{R}^N$ ($N\geq1$ ), $T>0$ is given and we set $Q_T=\Omega\times(0,T)$, and $$\mbox{$b:\Omega\times \mathbb{R}\to \mathbb{R}  \mbox{ is a
		Carath\'eodory function},$ }$$such that for every $x\in \Omega$, $b(x,.)$ is a strictly increasing
$C^1$-function with $b(x,0)=0$. Next, for any $k>0$, there exists
$\lambda_k>0$ and functions $A_k\in L^\infty(\Omega)$ and $B_k\in
L^p(\Omega)$ such that
\begin{equation}
\lambda_k\leq\frac{\partial b(x,s)}{\partial s}\leq A_k(x)
\mbox{ and } \Big|\nabla_x\Big(\frac{\partial b(x,s)}{\partial
	s}\Big)\Big|\leq B_k(x), \label{D2.2}
\end{equation}
for almost every $x\in \Omega$, for every $s$ such that $|s| \leq k$, we denote by $\nabla_x\big(\frac{\partial b(x,s)}{\partial s}\big)$ the
gradient of $\frac{\partial b(x,s)}{\partial s}$ defined in the
sense of distributions. \par Let $a:Q_T\times \mathbb{R}\times\mathbb{R}^N\to \mathbb{R}^N$ be a Carath\'eodory function, such that
\begin{equation}
|a(x, t,s, \xi)|\leq \beta [ k(x,t)  +|s|^{p-1}+  |\xi|^{p-1}]
,\label{D2.3}
\end{equation}
for a. e. $(x,t)\in Q_T$, all $(s,\xi)\in \mathbb{R}\times\mathbb{R}^N$, some positive function $k(x,t)\in L^{p'}(Q_T)$ and $\beta>0$.
\begin{gather}
[a(x,t,s,\xi) -a(x,t,s,\eta)](\xi-\eta) > 0   \mbox{ for all }
(\xi,\eta) \in \mathbb{R}^N\times\mathbb{R}^N,\ \mbox{with}\ \xi\neq\eta, \label{D2.4} \\
a(x, t,s, \xi) \xi \geq \alpha |\xi|^p,\mbox{ where $\alpha$ is a strictly positive constant.}
\label{D2.5}
\end{gather}
\par Furthermore, let $g(x,t,s,\xi): Q_T\times\mathbb{R}\times\mathbb{R}^N \to \mathbb{R}$ and $H(x,t,\xi): Q_T\times\mathbb{R}^N \to \mathbb{R}$ are two Carath\'eodory functions which satisfy, for almost every $(x,t)\in Q_T$ and for all $s\in \mathbb{R},\ \xi\in \mathbb{R}^N$, the following conditions
\begin{eqnarray}
&|g(x, t,s, \xi)|\leq L_1(|s|)(L_2(x,t)+|\xi|^p), \label{D2.6}\\
&g(x, t,s, \xi)s\geq0,\label{D2.66}
\end{eqnarray}
where $L_1:\mathbb{R}^+\to \mathbb{R}^+$ is a continuous increasing function, while $L_2(x,t)$ is positive and belongs to $L^1(Q_T)$.
\begin{equation}
\begin{array}{lll}
\exists\quad \delta>0, \quad\nu >0: &\mbox{for} \quad|s|\geq\delta,& |g(x,t,s,\xi)|\geq\nu|\xi|^{p},\end{array}
\label{Dg.3}\end{equation}
\begin{eqnarray}
|H(x,t,\xi)|\leq h(x,t)|\xi|^{p-1},\mbox{ where $h(x,t)$ is positive and belongs to $L^{p}(Q_T)$.}\label{D2.666}
\end{eqnarray}
We recall that, for $k>1$ and $s$ in $\mathbb{R}$, the truncation is defined as $T_k(s)= \max (-k, \min (k,s)).$
\par We shall use the following definition of renormalized solution for
problem \eqref{D1.1} in the following sense :
\begin{definition}\label{Ddef3.1}  Let $f\in L^1(Q_T)$ and $b(\cdot,u_0(\cdot))\in L^1(\Omega)$. A renormalized solution of problem \eqref{D1.1} is a function $u$ defined on $Q_T$, satisfying the following conditions:
	\begin{gather}
	T_k(u)\in L^p(0,T;W^{1,p}_0(\Omega))  \mbox{ for all $k\geq 0$
		and $b(x,u)\in L^{\infty}(0,T;L^1(\Omega))$,} \label{D2.7}\\
	\int_{\{m\leq |u|\leq m+1\}} a(x,t,u,\nabla u)\nabla u \,dx\,dt\to 0
	\mbox{ as }m\to +\infty, \label{D2.8}\\
	\begin{aligned}
	&\frac{\partial B_S(x, u)}{\partial t}-\operatorname{div }\Big(
	S'(u)
	a(x,t,u,\nabla u)\Big)+S''(u)a(x,t,u,\nabla u)\nabla u\\
	& \qquad\quad+g(x,t,u,\nabla u)S'(u)+H(x,t,\nabla u)S'(u)=fS'(u) \mbox{ in } \mathcal{D}'(Q_T),
	\end{aligned}\label{D2.9}
	\end{gather}
	for all functions $S\in W^{2, \infty}(\mathbb{R})$ which are piecewise $\mathcal{C}^1(\mathbb{R})$, such that $S'$ has a compact support in $\mathbb{R}$ and
	\begin{equation}\label{D2.10}
	B_S(x,u)(t=0)=B_S(x,u_0) \mbox{ in } \Omega,\mbox{ where $\displaystyle B_S(x,z)=\int_0^z \frac{\partial
			b(x,r)}{\partial r}S'(r)dr$.}
	\end{equation}
\end{definition}
\begin{remark}\label{Drmk5.2}
	Equation \eqref{D2.9} is formally obtained through pointwise multiplication of \eqref{D1.1} by $S'(u)$. However, while
	$a(x,t,u,\nabla u)$, $g(x,t,u,\nabla u)$ and $H(x,t,\nabla u)$ does not in general make sense in $\mathcal{D}'(Q_T)$, all the terms in \eqref{D2.9} have a meaning in $\mathcal{D}'(Q_T)$.
	\par Indeed, if $M$ is such that $\mbox{supp} S'\subset [-M,M]$, the following identifications are made in \eqref{D2.9} :
	\begin{itemize}
		\item[$\bullet$] $\displaystyle |B_S(x,u)|=|B_S(x,T_M(u))|\leq M \|S'\|_{L^\infty(\mathbb{R})}A_M(x)$ belongs to $L^\infty(\Omega)$ since $A_M$ is a bounded function.
		\item[$\bullet$]  $S'(u)a(x,t,u,\nabla u)$ identifies with $S'(u)a(x,t,T_M(u),\nabla T_M(u))$ a. e. in $Q_T$. Since $|T_M(u)|\leq
		M$ a. e. in $Q_T$ and $S'(u)\in L^{\infty}(Q_T)$, we obtain from \eqref{D2.3} and \eqref{D2.7} that
		\begin{equation*}
		S'(u)a(x,t,T_M(u),\nabla T_M(u))\in (L^{p'}(Q_T))^N.
		\end{equation*}
		\item[$\bullet$] $S''(u)a(x,t,u,\nabla u)\nabla u$ identifies with
		$S''(u)a(x,t,T_M(u),\nabla T_M(u))\nabla T_M(u)$ and
		\begin{equation*}
		S''(u)a(x,t,T_M(u),\nabla T_M(u))\nabla T_M(u)\in L^1(Q_T).
		\end{equation*}
		\item[$\bullet$] $S'(u)\Big(g(x,t,u,\nabla u)+H(x,t,\nabla u)\Big)$ identifies with\\
		$S'(u)\Big(g(x,t,T_M(u),\nabla T_M(u))+H(x,t,\nabla T_M(u))\Big)$ a. e. in $Q_T$. Since\\ $|T_M(u)|\leq M$ a. e.
		in $Q_T$ and $S'(u)\in L^{\infty}(Q_T)$, we obtain from \eqref{D2.3}, \eqref{D2.6} and \eqref{D2.666} that
		\begin{equation*}
		S'(u)\Big(g(x,t,T_M(u),\nabla T_M(u))+H(x,t,\nabla T_M(u))\Big)\in L^1(Q_T).
		\end{equation*}
		\item[$\bullet$] $S'(u)f$ belongs to $L^1(Q_T)$.
	\end{itemize}
	The above considerations show that \eqref{D2.9} holds in $\mathcal{D}'(Q_T)$ and	that
	\begin{equation}\label{Db15}
	\frac{\partial B_S(x,u)}{\partial t}\in L^{p'}(0,T; W^{-1,\ p'}(\Omega))+L^1(Q_T).
	\end{equation}
	The properties of $S$, assumptions \eqref{D2.2} and \eqref{D2.8} imply that
	\begin{equation}\label{Db16}
	\Big|\nabla B_S(x,u)\Big|\leq \|A_M\|_{L^\infty(\Omega)}|\nabla T_M(u)|
	\|S'\|_{L^\infty(\mathbb{R})}+ M \|S'\|_{L^\infty(\mathbb{R})}B_M(x),
	\end{equation}
	and
	\begin{equation}\label{Db17}
	B_S(x,u)\  \mbox{ belongs to }\ L^p(0,T;W^{1,p}_0(\Omega)).
	\end{equation}
	Then (\ref{Db15}) and (\ref{Db17}) imply that $B_S(x,u)$ belongs to
	$C^0([0,T];L^1(\Omega))$ (for a proof of this trace result see
	~\cite{Po1}), so that the initial condition \eqref{D2.10} makes sense.
	\par Also remark that, for every $S\in W^{1,\infty}(\mathbb{R})$, nondecreasing function such that supp $S'\subset [-M,M]$, in view of \eqref{D2.2} we have
	\begin{equation*}
	\lambda_M|S(r)-S(r')|\leq \Big|B_S(x,r)-B_S(x,r')\Big|\leq
	\|A_M\|_{L^\infty(\Omega)}|S(r)-S(r')|,\label{D2.11} \mbox{ a. e. } x\in \Omega, \, \forall \,r, r'\in \mathbb{R}.
	\end{equation*}
\end{remark}
\section{Statements of results}
\quad The main results of this article are stated as follows.
\begin{theorem} \label{Dthm1}
	Let $f\in L^1(Q_T)$ and $u_0$ is a measurable function such that $b(\cdot,u_0)\in L^1(\Omega)$. Assume that \eqref{D2.2}--\eqref{D2.666} hold true. Then, there exists a renormalized solution $u$ of  problem \eqref{D1.1} in the sense of Definition \ref{Ddef3.1}.
\end{theorem}
\begin{proof} The proof of Theorem \ref{Dthm1} is done in five steps.
	\subsection*{Step 1 : Approximate problem and a priori estimates}
	\quad For $n>0$, let us define the following approximation of $b,\, f$ and $u_0$.\par First, set
	$ b_n(x,r)=b(x,T_n(r))+\frac{1}{n}r.$ $b_n$ is a Carath\'eodory function and satisfies \eqref{D2.2}, there exist $\lambda_n>0$ and functions $A_n\in L^\infty(\Omega)$ and $B_n\in L^p(\Omega)$ such that $\lambda_n\leq \frac{\partial b_n(x,s)}{\partial s}\leq A_n(x)$ and $\big|\nabla_x \Big(\frac{\partial b_n(x,s)}{\partial s}\Big)\big|\leq B_n(x),$ a. e. in $\Omega$, $s\in\mathbb{R}$. \par Next, set
	$g_{n}(x, t, s, \xi)=\frac{g(x, t, s, \xi)}{1+\frac{1}{n}|g(x, t, s, \xi)|} \mbox{ and } H_{n}(x, t, \xi)=\frac{H(x, t,\xi)}{1+\frac{1}{n}|H(x, t, \xi)|}.$
	\par Note that
	$|g_{n}(x, t, s, \xi)| \leq \max \{|g(x, t, s, \xi)|; n \} \mbox{ and } |H_{n}(x, t, \xi)| \leq \max \{|H(x, t, \xi)|; n \}.$
	\par Moreover, since $f_n\in L^{p'}(Q_T)$ and $ f_n\rightarrow f$ a. e. in $Q_T$ and strongly in $L^1(Q_T)$ as $n\rightarrow\infty.$
	\begin{equation}
	u_{0n}\in \mathcal{D}(\Omega),\ b_n(x,u_{0n})\rightarrow b(x,u_{0}) \mbox{ a. e. in }  \Omega \mbox{ and strongly in } L^1(\Omega) \mbox{ as } n\rightarrow\infty.\label{Du0n}
	\end{equation}
	Let us now consider the approximate problem
	\begin{equation}\left\{
	\begin{array}{llll}
	\displaystyle {\partial b_n(x,u_n) \over \partial t} - \mbox{div}
	(a(x,t,u_n,\nabla u_n))+g_{n}(x,t,u_n,\nabla u_n) +H_{n}(x,t,\nabla u_n)
	&=& f_n &\mbox{  in }\ Q_T,\\
	b_n(x,u_n)(t=0)&=&b_n(x,u_{0n})&\mbox{  in } \Omega,\\
	u_{n}&=&0&\mbox{  in } \partial\Omega\times(0,T).
	\end{array} \right. \label{Dpn}
	\end{equation}
	Since $f_n\in L^{p'}(0,T;W^{-1,\ p'}(\Omega))$, proving existence of a weak solution  $u_n\in $ $L^p(0,T;W^{1,\ p}_0(\Omega))$
	of \eqref{Dpn} is an easy task (see e.g. \cite[p.~271]{Li}), i. e.
	\begin{align*}
	&\displaystyle \int_0^T\langle \displaystyle {\partial b_n(x,u_n) \over \partial t},v\rangle dt+\int_{Q_T} a(x,t,u_n,\nabla u_n)\nabla v\,dx\,dt +\int_{Q_T} g_{n}(x,t,u_n,\nabla u_n)v\,dx\,dt \\
	&+\int_{Q_T} H_{n}(x,t,\nabla u_n)v\,dx\,dt=\int_{Q_T} f_nv\,dx\,dt,\mbox{ for all $v \in  L^{p}(0, T; W^{1,p}(\Omega)) \cap L^{\infty}(Q_{T})$}.
	\end{align*}
	\par Now, we prove the solution $u_n$ of problem \eqref{Dpn} is bounded in $L^p(0,T;W^{1,\ p}_0(\Omega))$,
	\begin{lemma}Let $u_n \in L^p(0,T;W^{1,\ p}_0(\Omega))$ be a weak solution of \eqref{Dpn}. Then, the following estimates holds,
		\begin{eqnarray}||u_n ||_{L^p(0,T;W^{1,\ p}_0(\Omega))}\leq D ,\label{Dbounded}\end{eqnarray}
		where $D$ depend only on $\Omega,$ $T$, $N$, $p$, $p'$, $f$ and $||h||_{L^{p}(Q_T)}$.\label{Dthm.36}\end{lemma}
	\begin{proof} To get \eqref{Dbounded}, we divide the integral $\displaystyle{\int_{Q_T}}|\nabla u_n|^{p}\,dx\,dt$ in two parts and we prove the following estimates : for all $k\geq 0$
		\begin{eqnarray}
		&\displaystyle{\int_{\{|u_n|\leq k\}}}|\nabla u_n|^{p} \,dx\,dt\leq M_1  k,\label{Deq.im318}\\
		&\mbox{ and } \displaystyle{\int_{\{|u_n|> k\}}}|\nabla u_n|^{p} \,dx\,dt \leq M_2, \label{Deq.im319}
		\end{eqnarray}
		where $M_1$ and $M_2$ are positive constants. In what follows we will denote by $M_i$, $i=3,4,...$, some generic positive constants. We suppose $p<N$, (the case $p\geq N$ is similar). For $\varepsilon>0$ and $s\geq 0$, we define
		$\displaystyle \varphi_\varepsilon(r)= \left\{\begin{array}{lcl}
		\mbox{sign}(r)                                           & \mbox{if} & |r| >s+\varepsilon , \\
		\displaystyle\frac{\mbox{sign}(r)(|r|-s ) }{\varepsilon} & \mbox{if} & s<|r| \leq s+\varepsilon , \\
		0 & \         & \mbox{otherwise}.
		\end{array}\right.$
		We choose $v= \varphi_\varepsilon(u_n)$ as test function in \eqref{Dpn}, we have
		\begin{align*}
		&\left[\int_{\Omega} B_{\varphi_\varepsilon}^n(x,u_n)\,dx\right]_0^T+\int_{Q_T} a(x,t,u_n,\nabla u_n)\nabla ( \varphi_\varepsilon(u_n))\,dx\,dt\\
		&+\int_{Q_T} g_{n}(x,t,u_n,\nabla u_n)\varphi_\varepsilon(u_n) \,dx\,dt +\int_{Q_T} H_{n}(x,t,\nabla u_n)\varphi_\varepsilon(u_n)\,dx\,dt\\
		&=\int_{Q_T} f_n\varphi_\varepsilon(u_n)\,dx\,dt,\mbox{ where $B_{\varphi_\varepsilon}^n(x,r)=\displaystyle{\int_0^r\frac{\partial b_n(x,s)}{\partial s} \varphi_\varepsilon(s)\,ds}$.}\end{align*}
		\par Using $B_{\varphi_\varepsilon}^n(x,r)\geq0$, $g_{n}(x,t,u_n,\nabla u_n)\varphi_\varepsilon(u_n)\geq0$, \eqref{D2.5}, \eqref{D2.666}, H\"{o}lder inequality and Letting $\varepsilon$ go to zero, we obtain
		\begin{equation*} \begin{array}{l}
		\displaystyle{\frac{-d}{ds}\int_{\{s<|u_n|\}} \alpha|\nabla u_n|^p\,dx\,dt}\leq\displaystyle{\int_{\{s<|u_n|\}} |f_n|\,dx\,dt}\\
		\quad+\displaystyle{\int_s^{+\infty}\left(\frac{-d}{d\sigma}\int_{\{\sigma<|u_n|\}}h^p\,dx\,dt\right)^{\frac{1}{p}}\left(\frac{-d}{d\sigma}\int_{\{\sigma<|u_n|\}}|\nabla u_n|^{p}\,dx\,dt\right)^{\frac{1}{p'}}\,d\sigma},\end{array}\end{equation*}
		where $\{s<|u_n|\}$ denotes the set $\{(x,t)\in Q_T, s<|u_n(x,t)|\}$ and $\mu(s)$ stands for the distribution function of $u_n$, that is $\mu(s)=|\{(x,t)\in Q_T,\ |u_n(x,t)|>s\}|$ for all $s\geq0$.
		\par On the other hand, from Fleming-Rishel coarea formula and isoperimetric inequality, we have for almost every $s>0$
		\begin{equation}
		N C_N^{\frac{1}{N}}(\mu(s))^{\frac{N-1}{N}}\leq-\frac{d}{d s}\displaystyle{\int_{\{s<|u_n| \}}}|\nabla u_n|^{p}dx\,dt,\label{Equation-Rev1}
		\end{equation}
		where $C_N$ is the measure of the unit ball in $\mathbb{R}^N$. Using the H\"{o}lder's inequality we obtain that for almost every $s>0$
		\begin{eqnarray}
		-\frac{d}{d s}\displaystyle{\int_{\{s<|u_n| \}}}|\nabla u_n|^{p}dx\,dt\leq(-\mu'(s))^{\frac{1}{p'}} \left(-\frac{d}{d s}\displaystyle{\int_{\{s<|u_n| \}}}|\nabla u_n|^{p}dx\,dt\right)^{\frac{1}{p}}.\label{Equation-Rev2}
		\end{eqnarray}
		Then, combining \eqref{Equation-Rev1} and \eqref{Equation-Rev2} we obtain for almost every $s>0$
		\begin{eqnarray}
		1 \leq \left(N C_N^{\frac{1}{N}}\right)^{-1}(\mu(s))^{\frac{1}{N}-1}(-\mu'(s))^{\frac{1}{p'}} \left(-\frac{d}{d s}\displaystyle{\int_{\{s<|u_n| \}}}|\nabla u_n|^{p}dx\,dt\right)^{\frac{1}{p}}.\label{Deq.im310}\end{eqnarray}
		Using (\ref{Deq.im310}), we have
		\begin{equation}\label{Deq.im311}\begin{array}{l}
		\alpha\left(\displaystyle{\frac{-d}{ds}\int_{\{s<|u_n|\}} |\nabla u_n|^p\,dx\,dt}\right)^{\frac{1}{p'}}\\ \leq\left(NC_N^{\frac{1}{N}}\right)^{-1}(\mu(s))^{\frac{1}{N}-1}(-\mu'(s))^{\frac{1}{p'}}\left(\displaystyle{\int_{\{s<|u_n|\}}|f_n|\,dx\,dt}\right)
		\\
		\quad\quad\quad+\left(NC_N^{\frac{1}{N}}\right)^{-1}(\mu(s))^{\frac{1}{N}-1}(-\mu'(s))^{\frac{1}{p'}}\\ \quad\quad\quad\times\displaystyle{\int_s^{+\infty}\left(\frac{-d}{d\sigma}\int_{\{\sigma<|u_n|\}}h^p\,dx\,dt\right)^{\frac{1}{p}}\left(\frac{-d}{d\sigma}\int_{\{\sigma<|u_n|\}}|\nabla u_n|^{p}\,dx\,dt\right)^{\frac{1}{p'}}\,d\sigma}.\end{array}\end{equation}
		Now, we consider two functions $B$ and $\psi$ (see Lemma 2.2 of \cite{A.T}) defined by
		\begin{eqnarray}
		&\displaystyle{ \int_{\{s<|u_n| \}}}h^p(x,t)\,dx\,dt= \int_{0}^{\mu(s)}B^p(\sigma)d \sigma.\label{Deq.im312}\\
		&\mbox{and }\psi(s)= \displaystyle{\int_{\{s<|u_n| \}}}  | f_n| \,dx\,dt.\label{Deq.im313}
		\end{eqnarray}
		We have $\label{Deq.im314}||B||_{L^p(0,T;W^{1,p}_0(\Omega))}\leq||h||_{L^p(0,T;W^{1,p}_0(\Omega))}\mbox{ and }|\psi(s)|\leq || f_n||_{L^1(Q_T)}. $
		From (\ref{Deq.im311}), (\ref{Deq.im312}) and (\ref{Deq.im313}) we have
		$$\begin{array}{l}
		\alpha\left(\frac{-d}{ds}\displaystyle{\int_{\{s<|u_n|\}}}|\nabla u_n|^{p}dxdt\right)^{\frac{1}{p'}} \leq  \\
		\quad (N C_N^{\frac{1}{N}})^{-1}(\mu(s))^{\frac{1}{N}-1}  (-\mu'(s))^{\frac{1}{p'}}\psi(s)+(N C_N^{\frac{1}{N}})^{-1}(\mu(s))^{\frac{1}{N}-1}\\
		\quad \times(-\mu'(s))^{\frac{1}{p'}}
		\displaystyle{\int_{s}^{+\infty}}B(\mu(\nu))(-\mu'(\nu))^{\frac{1}{p}}\left(-\frac{d}{d \nu}\displaystyle{\int_{\{\nu<|u_n| \}}}|\nabla u_n|^{p}dxdt\right)^{\frac{1}{p'}}d\nu.
		\end{array}$$
		From Gronwall's Lemma (see \cite{B.B}), we obtain
		\begin{equation}\begin{array}{l}
		\alpha\left(\frac{-d}{ds}\displaystyle{\int_{\{s<|u_n|\}}}|\nabla u_n|^{p}\,dx\,dt\right)^{\frac{1}{p'}}
		\leq \\
		(N C_N^{\frac{1}{N}})^{-1}(\mu(s))^{\frac{1}{N}-1}  (-\mu'(s))^{\frac{1}{p'}}\psi(s)+(NC_N^{\frac{1}{N}})^{-1}(\mu(s))^{\frac{1}{N}-1}\\
		\times(-\mu'(s))^{\frac{1}{p'}}\displaystyle{\int_{s}^{+\infty}}\Big[(N C_N^{\frac{1}{N}})^{-1}(\mu(\sigma))^{\frac{1}{N}-1} \psi(\sigma)\Big] B(\mu(\sigma))(-\mu'(\sigma))\\
		\times\mbox{ exp} \left(\displaystyle{\int_{s}^{\sigma}}(N C_N^{\frac{1}{N}})^{-1}) B(\mu(r))(\mu(r))^{\frac{1}{N}-1}(-\mu'(r))dr\right)d \sigma.
		\end{array}\label{Deq.im324}
		\end{equation}
		Now, by a variable of change and by H\"{o}lder inequality, we estimate the argument of the exponential function on the right hand side of (\ref{Deq.im324})
		\begin{equation*}
		\begin{array}{lll}
		\displaystyle{\int_{s}^{\sigma}}B(\mu(r))(\mu(r))^{\frac{1}{N}-1}(-\mu'(r))dr&=&\displaystyle{\int_{s}^{\sigma}}B(z)z^{\frac{1}{N}-1}dz\\
		\ &\leq&\displaystyle{\int_{0}^{|\Omega|}}B(z)z^{\frac{1}{N}-1}dz\\
		\ &\leq &||B||_{L^p}\left(\displaystyle{\int_{0}^{|\Omega|}}z^{(\frac{1}{N}-1)p'}\right) ^{\frac{1}{p'}}.
		\end{array}
		\label{Deq.im325}
		\end{equation*}
		Raising to the power $p'$ in \eqref{Deq.im324} and we can write
		\begin{eqnarray*}
			\frac{-d}{ds}\displaystyle{\int_{\{s<|u_n|\}}}|\nabla u_n|^{p}\,dx\,dt\leq M_1.\label{Deq.im326}
		\end{eqnarray*}
		where $M_1$ depend only on $\Omega,$ $N$, $p$, $p'$, $f$, $\alpha$ and $||h||_{L^{p}(Q_T)}$, integrating between $0$ and $k$, \eqref{Deq.im318} is proved.
		\par We now give the proof of \eqref{Deq.im319}, using $T_k(u_n)$ as test function in \eqref{Dpn}, gives
		$$
		\begin{array}{lll}
		\displaystyle{\left[\int_{\Omega} B_{k}^n(x,u_n)\,dx\right]_0^T}+\displaystyle{\int_{\Omega}a(x,t,u_n,\nabla u_n)\nabla T_k(u_n)\,dx\,dt} \\
		\qquad\quad  +\displaystyle{ \int_{\Omega}(g_{n}(x,t,u_n,\nabla u_n)+H_{n}(x,t,\nabla u_n)) T_k(u_n)\,dx\,dt}\\
		\qquad\quad = \displaystyle{\int_{\Omega}f_n T_k(u_n)\,dx\,dt },
		\end{array}
		$$
		where $B_{k}^n(x,r)=\displaystyle{\int_0^r\frac{\partial b_n(x,s)}{\partial s} T_k(s)\,ds}$. Using \eqref{D2.666}, we deduce that,\begin{equation*}
		\begin{array}{l}\displaystyle{\left[\int_{\Omega} B_{k}^n(x,u_n)\,dx\right]_0^T}+\displaystyle{\int_{\{|u_n| \leq k\}}a(x,t,u_n,\nabla u_n)\nabla u_n\,dx\,dt} \\
		+\displaystyle{ \int_{\{|u_n| \leq k\}}g_{n}(x,t,u_n,\nabla u_n)u_n\,dx}+\displaystyle{ \int_{\{|u_n| > k\}}g_{n}(x,t,u_n,\nabla u_n)T_k(u_n)\,dx\,dt}\\
		\leq \displaystyle{\int_{\Omega}f_n T_k(u_n)\,dx\,dt +\displaystyle{\int_{\Omega}}h(x,t)|\nabla u_n|^{p-1}|T_k(u_n)|\,dx\,dt,}\end{array}\label{Deq.im327}
		\end{equation*} and by using the fact that $B_{k}^n(x,r)\geq0$, $g_{n}(x,t,u_n,\nabla u_n) u_n\geq 0$ and \eqref{D2.5}, we have
		$$
		\begin{array}{lll}
		\alpha\displaystyle{\int_{\{|u_n| \leq k\}}}|\nabla u_n|^{p}\,dx\,dt+\displaystyle{\int_{\{|u_n| > k\}}}g(x,u_n,\nabla u_n)T_k(u_n)\,dx\,dt\\
		\quad\quad\leq k||f||_{L^1} +k\displaystyle{\int_{\{|u_n| \leq k\}}}h(x,t)|\nabla u_n|^{p-1}\,dx\,dt\\
		\quad\quad\quad+k\displaystyle{\int_{\{|u_n| \geq k\}}}h(x,t)|\nabla u_n|^{p-1}\,dx\,dt.
		\end{array}$$
		By H\"{o}lder inequality and \eqref{Deq.im318}, \eqref{Dg.3} and applying Young's inequality, we get for all $k>\delta$
		\begin{equation*}
		\begin{array}{l}
		\nu k \displaystyle{\int_{\{|u_n| > k\}}}|\nabla u_n|^{p}\,dx\,dt\leq k||f||_{L^1(Q_T)} +k^{1+\frac{1}{p'}}M_1||h ||_{L^{p}Q_T)}+k\displaystyle{\int_{\{|u_n| > k\}}}h(x,t)|\nabla u_n|^{p-1}\,dx\,dt\\
		\leq k||f||_{L^1(Q_T)}+k^{1+\frac{1}{p'}}M_1||h ||_{L^{p}Q_T)}
		+M_6 k||h||_{L^p}^p+\frac{1}{p'}\nu k\displaystyle{\int_{\{|u_n| > k\}}}|\nabla u_n|^{p}\,dx\,dt.
		\end{array} \label{Deq.im328}
		\end{equation*}
		Hence
		\begin{equation}
		(1-\frac{1}{p'})\displaystyle{\int_{\{|u_n| > k\}}}|\nabla u_n|^{p}\,dx\,dt
		\leq M_3 ||f||_{L^1(Q_T)}+k^{\frac{1}{p'}}M_5||h ||_{L^{p}(Q_T)}+M_7||h||_{L^p}^p,
		\label{Deq.im331}
		\end{equation}
		and Lemma \ref{Dthm.36} is proved.
	\end{proof}
	Then there exists $u\in L^p(0,T;W^{1,p}_0(\Omega))$ such that, for some subsequence
	\begin{equation}\begin{array}{l}
	u_n \rightharpoonup u \mbox{ weakly in } L^p(0,T;W^{1,p}_0(\Omega)),
	\end{array}\label{D24}\end{equation}
	we conclude that
	\begin{equation}
	||T_k(u_n)||^p_{L^p(0,T;W^{1,p}_0(\Omega))}\leq c_2k. \label{Dtk}
	\end{equation}
	We deduce from the above inequalities, \eqref{D2.2} and \eqref{Dtk},
	that
	\begin{equation}
	\int_{\Omega} B _k^n(x,u_n)dx\leq  Ck,\label{D514}
	\end{equation}
	where $\displaystyle B _k^n(x,z)= \int_0^z {\partial b_n(x,s)\over \partial s} T_k(s)\,ds.$
	\par Now, we turn to prove the almost every convergence of $u_n$ and
	$b_n(x,u_n)$. Consider now a function non decreasing $\xi_k\in
	C^2(\mathbb{R})$ such that $\xi_k(s)=s$  for $|s|\leq \frac{k}{2}$ and
	$\xi_k(s)=k$ for $|s|\geq k $. Multiplying the approximate equation by
	$\xi'_k(u_n)$, we obtain
	\begin{equation}
	\displaystyle \frac{\partial B^n_\xi(x,u_n)}{\partial t}-\operatorname{div}
	\Big(a(x,t,u_n,\nabla u_n)\xi'_k(u_n)\Big)
	+a(x,t,u_n,\nabla u_n)\xi''_k(u_n)\nabla u_n\label{D3.15}
	\end{equation}
	$$\begin{array}{l}\displaystyle+\Big(g_{n}(x,t,u_n,\nabla u_n)+H_{n}(x,t,\nabla u_n)\Big)\xi'_k(u_n)\\
	\quad\quad=f_n \xi'_k(u_n),\end{array}$$
	in the sense of distributions, where $ B^n_\xi(x,z)=\displaystyle\int_0^z\frac{\partial b_n(x,s)}{\partial s}\xi'_k(s) ds$. As a consequence of \eqref{Dtk}, we deduce that $\xi_k(u_n)$ is bounded in $L^p(0,T;W^{1,p}_0(\Omega))$ and $\frac{\partial B_\xi^n(x,u_n)}{\partial t}$ is bounded in $L^1(Q_T)+L^{p'}(0,T;W^{-1,p'}(\Omega))$. Due to the properties of $\xi_k$ and \eqref{D2.2}, we conclude that $\displaystyle\frac{\partial \xi_k(u_n)}{\partial t}$ is bounded in $L^1(Q_T)+L^{p'}(0,T;W^{-1,p'}(\Omega))$, which implies that $\xi_k(u_n)$ strongly converges in $L^1(Q_T)$ (see \cite{Po1}).
	\par Due to the choice of $\xi_k$, we conclude that for each $k$, the sequence $T_k(u_n)$ converges almost everywhere in $Q_T$, which implies that $u_n$ converges almost everywhere to some measurable function $u$ in $Q_T$.
	Thus by using the same argument as in \cite{B.M,B.M.R} and \cite{Re3}, we can show
	\begin{equation}
	u_n\to u \mbox{ a. e. in }Q_T, \label{D3.16}\end{equation}\begin{equation*}
	b_n(x,u_n)\to b(x,u) \mbox{ a. e. in } Q_T.\label{D3.17}
	\end{equation*}
	We can deduce from \eqref{Dtk} that
	\begin{equation*}
	T_k(u_n)\rightharpoonup T_k(u) \mbox{ weakly in }L^p(0,T;W^{1,p}_0(\Omega)). \label{D3.18}
	\end{equation*}
	Which implies, by using \eqref{D2.3}, for all $k>0$ that there exists
	a function \\ $\overline{a}\in ( L^{p'}(Q_T))^N$, such that
	\begin{equation}
	a(x,t,T_k(u_n),\nabla T_k(u_n))\rightharpoonup \overline{a}   \mbox{ weakly in }
	( L^{p'}(Q_T))^N. \label{D3.19}
	\end{equation}
	\par We now establish that $b(.,u)$ belongs to
	$L^{\infty}(0,T;L^1(\Omega))$. Using \eqref{D3.16} and  passing to the
	limit-inf in \eqref{D514} as $n$ tends to $+\infty$, we obtain that
	$$
	\frac{1}{k}\int_{\Omega}B_k(x,u)(\tau)dx \leq C,
	$$
	for almost any $\tau$ in $(0,T)$. Due to the definition of
	$B_k(x,s)$ and the fact that $\frac{1}{k}B_k(x,u)$ converges
	pointwise to $b(x,u)$, as $k$ tends to $+\infty$, shows that
	$b(x,u)$ belong to $L^{\infty}(0,T;L^1(\Omega))$.
	\begin{lemma} \label{Dlem5.4}
		Let $u_n$ be a solution of the approximate problem \eqref{Dpn}. Then
		\begin{equation}\lim _{m\to
			\infty}\limsup_{n\to \infty}\int_{\{m\leq|u_n|\leq m+1\}}
		a(x,t,u_n,\nabla u_n)\nabla u_n\,dx\,dt=0 \label{D3.20}
		\end{equation}
	\end{lemma}
	\begin{proof} We use $T_1(u_n-T_m(u_n))^+=\alpha_m(u_n)\in L^p(0,T;W^{1,p}_0(\Omega))\cap L^\infty(Q_T)$ as test function in \eqref{Dpn}. Then, we have
		\begin{align*}
		& \int_0^T\langle {\partial b_n(x,u_n)\over \partial t}\ ;\ \alpha_m(u_n)\rangle
		\,dt +\int_{\{m\leq u_n\leq m+1\}} a(x,t,u_n,\nabla u_n)\nabla u_n
		\alpha'_m(u_n)
		\,dx\,dt \\
		&\quad +\int_{Q_T}\Big(g_{n}(x,t,u_n,\nabla u_n)+H_{n}(x,t,\nabla u_n)\Big)  \alpha_m(u_n) \,dx\,dt\\
		&\quad\quad\leq  \int_{Q_T}|f_n  \alpha_m(u_n)| \,dx\,dt.
		\end{align*}
		Which,  by setting $\displaystyle B^n_{m}(x,r)=\int_0^r \frac{\partial
			b_n(x,s)}{\partial s}\alpha_m(s) ds$, \eqref{D2.66} and \eqref{D2.666} gives
		\begin{align*}
		&\int_{\Omega}B^n_{m}(x,u_n)(T)dx +\int_{\{m\leq u_n\leq
			m+1\}} a(x,t,u_n,\nabla u_n)\nabla u_n  \,dx\,dt \\
		& \leq \int_{\{m\leq u_n\}}|f_n | \,dx\,dt++\int_{Q_T} h(x,t)|\nabla u_n|^{p-1}\,dx\,dt .
		\end{align*}
		Now we use H\"{o}lder's inequality and \eqref{Dbounded}, in order to deduce
		\begin{align*}
		&\int_{\Omega}B^n_{m}(x,u_n)(T)dx + \int_{\{m\leq u_n\leq
			m+1\}} a(x,t,u_n,\nabla u_n)\nabla u_n  \,dx\,dt \\
		& \leq \int_{\{m\leq u_n\}}|f_n | \,dx\,dt+c_1\left(\int_{\{m\leq u_n\}}|h(x,t)|^p\,dx\,dt\right)^{\frac{1}{p'}}.
		\end{align*}
		Since $B^n_{m}(x,u_n)(T)\geq 0$ and the strong convergence of $f_n$ in $L^1(Q_T)$, by Lebesgue's theorem, we have
		\begin{equation*}\lim _{m\to
			\infty}\lim_{n\to \infty}\int_{\{m\leq u_n\}}|f_n | \,dx\,dt=0.\label{Dfexp}
		\end{equation*}
		Similarly, since $h \in L^p(Q_T)$, we obtain
		\begin{equation*}\lim _{m\to
			\infty}\lim_{n\to \infty}\left(\int_{\{m\leq u_n\}}|h(x,t)|^p\,dx\,dt\right)^{\frac{1}{p'}}=0.\label{Dgexp}
		\end{equation*}
		We conclude that
		\begin{equation}\lim _{m\to
			\infty}\limsup_{n\to \infty}\int_{\{m\leq u_n\leq m+1\}}
		a(x,t,u_n,\nabla u_n)\nabla u_n\,dx\,dt=0 .\label{Dan1}
		\end{equation}
		On the other hand, using $T_1(u_n-T_m(u_n))^-$ as test
		function in \eqref{Dpn} and reasoning as in the proof of
		\eqref{Dan1} we deduce that
		\begin{equation}
		\lim _{m\to \infty}\limsup_{n\to \infty}\int_{\{-(m+1)\leq u_n\leq
			-m\}} a(x,t,u_n,\nabla u_n)\nabla u_n\,dx\,dt=0 .\label{Dan2}
		\end{equation}
		Thus \eqref{D3.20} follows from \eqref{Dan1} and \eqref{Dan2}.
	\end{proof}
	\subsection*{Step 2 : Almost everywhere convergence of the gradients}
	\quad This step is devoted to introduce for $k\geq 0$ fixed a time
	regularization of the function $T_k(u)$ in order to perform the
	monotonicity method (the proof of this steps is similar the Step 4 in \cite{B.M.R}). This kind of regularization has been first introduced by R. Landes (see Lemma 6 and Proposition 3 \cite[p. 230]{La}  and Proposition 4, \cite[p. 231]{La}). For $k> 0$ fixed, and let $\varphi(t)=te^{\gamma t^2}$, $\gamma>0$. It is well known that when $\gamma>\left(\frac{L_1(k)}{2\alpha}\right)^2$, one has
	\begin{equation}\begin{array}{l}
	\varphi'(s)-(\frac{L_1(k)}{\alpha})|\varphi(s)|  \geq  \frac{1}{2},\mbox{ for all }s \in \mathbb{R}. \end{array}\label{D26}\end{equation}
	Let $\{\psi_i\} \subset \mathcal{D}(\Omega)$ be a sequence which converge strongly to
	$u_0$ in $L^1(\Omega)$.\\ Set $w_{\mu}^i=(T_k(u))_{\mu}+e^{-\mu
		t}T_k(\psi_i)$ where $(T_k(u))_{\mu}$ is the mollification with
	respect to time of $T_k(u)$. Note that $w_{\mu}^i$ is a smooth
	function having the following properties:
	\begin{gather}
	\frac{\partial w_{\mu}^i}{\partial t}=\mu(T_k(u)-w_{\mu}^i),\quad
	w_{\mu}^i(0)=T_k(\psi_i),\quad \big|w_{\mu}^i\big|\leq k,\\
	w_{\mu}^i\to T_k(u)  \mbox{ strongly in }L^p(0,T;W^{1,p}_0(\Omega)),  \mbox{ as } \mu\to \infty.\label{D3.25}
	\end{gather}
	We introduce the following function of one real variable:
	$$
	h_m(s)=\begin{cases}
	1 &  \mbox{if }  |s| \leq m,\\
	0  &  \mbox{if } |s|\geq m+1,\\
	m+1- |s|  &  \mbox{if } m\leq |s|\leq m+1,
	\end{cases}
	$$
	where $m>k$. Let $\theta_{n}^{\mu,i}=T_k(u_n)-w^{i}_{\mu}$ and $z_{n,m}^{\mu,i}=\varphi(\theta_{n}^{\mu,i})h_m(u_n)$. Using in \eqref{Dpn} the test function $z_{n,m}^{\mu,i}$, we obtain since $g_{n}(x,t,u_n,\nabla u_n)\varphi(T_k(u_n)-w^{i}_{\mu})h_m(u_n)\geq0$ on $\{|u_n|>k\}$:
	\begin{equation}\begin{array}{l}
	\displaystyle \int_0^T\langle \frac{\partial
		b_n(x,u_n)}{\partial t}\ ;\
	\varphi(T_k(u_n)-w_{\mu}^i)h_m(u_n)\rangle\,dt\\
	+ \displaystyle{\int_{Q_T}}a(x,t,u_n,\nabla u_n)(\nabla T_k(u_n)- \nabla w_{\mu}^i)\varphi'(\theta_{n}^{\mu,i})h_m(u_n)\,dx\,dt \\
	+\displaystyle{\int_{Q_T}}a(x,t,u_n,\nabla u_n)\nabla u_n\varphi(\theta_{n}^{\mu,i})h'_m(u_n)\,dx\,dt \\
	+\displaystyle{\int_{\{|u_n|\leq k\}}}g_{n}(x,t,u_n,\nabla u_n)\varphi(T_k(u_n)-w^{i}_{\mu})h_m(u_n)\,dx\,dt\\
	\quad\leq \displaystyle{\int_{Q_T}} |f_n z_{n,m}^{\mu,i}|\,dx\,dt +\displaystyle{\int_{Q_T}}|H_{n}(x,t,\nabla u_n)z_{n,m}^{\mu,i}|\,dx\,dt.\end{array}\label{D28} \end{equation}
	In the rest of this paper, we will omit for simplicity the denote $\varepsilon(n,\mu,i,m)$ all quantities (possibly different) such that 	 \begin{equation*}\displaystyle{\lim_{m\rightarrow \infty}\lim_{i\rightarrow \infty}\lim_{\mu\rightarrow \infty}\lim_{n\rightarrow \infty}\varepsilon(n,\mu,i,m)=0,}	\end{equation*}
	and this will be the order in which the parameters we use will tend to infinity, that is, first $n$, then $\mu,i$ and finally $m$. Similarly we will write only $\varepsilon(n)$, or $\varepsilon(n,\mu)$,... to mean that the limits are made only on the specified parameters.
	\par We will deal with each term of \eqref{D28}. First of all, observe that
	\begin{equation*}
	\displaystyle{\int_{Q_T}} |f_n z_{n,m}^{\mu,i}|\,dx\,dt +\displaystyle{\int_{Q_T}}|H_{n}(x,t,\nabla u_n)z_{n,m}^{\mu,i}|\,dx\,dt  =  \varepsilon(n,\mu),\label{D29} \end{equation*}
	since $\varphi(T_k(u_n)-w^{i}_{\mu})h_m(u_n)$ converges to $\varphi(T_k(u)-(T_k(u))_{\mu}+e^{-\mu t}T_k(\psi_i))h_m(u)$ strongly in $L^p(Q_T)$ and weakly$-*$ in $L^\infty(Q_T)$ as $n\rightarrow\infty$ and finally $\varphi(T_k(u)-(T_k(u))_{\mu}+e^{-\mu t}T_k(\psi_i))h_m(u)$ converges to $0$ strongly in $L^p(Q_T)$ and weakly$-*$ in $L^\infty(Q_T)$ as $\mu\rightarrow\infty$.
	Thanks to \eqref{D3.20} the third and fourth integrals on the right hand side of \eqref{D28} tend to zero as $n$ and $m$ tend to infinity,
	and by Lebesgue's theorem and  $F \in (L^{p'} (Q_T))^N$, we deduce that the right hand side of \eqref{D28} converges to zero as $n$, $m$ and
	$\mu$ tend to infinity. Since $(T_k(u_n)-w^i_\mu)h_m(u_n) \rightharpoonup (T_k(u)-w^i_\mu)h_m(u)$ weakly$^*$in $L^1(Q_T)$ and strongly in $L^p(0,T;W^{1,p}_0(\Omega))$ and $ (T_k(u)-w^i_\mu)h_m(u)\rightharpoonup 0$ weakly$^*$in $L^1(Q_T)$ and strongly in $L^p(0,T;W^{1,p}_0(\Omega))$ as $\mu\rightarrow+\infty.$
	\par On the one hand, the  definition of the sequence $w_{\mu}^i$ makes it possible to establish the following lemma.
	\begin{lemma}\label{Dznn}
		For $k\geq 0$ we have
		\begin{equation}
		\int_0^T\langle \frac{\partial
			b_n(x,u_n)}{\partial t}\ ;\ \varphi(T_k(u_n)-w_{\mu}^i)h_m(u_n)\rangle\,dt\geq \varepsilon(n,m,\mu,i).\label{Dzn}
		\end{equation}
	\end{lemma}
	\begin{proof}(see  Blanchard and Redwane \cite{B.M.R1}).
	\end{proof}
	On the other hand, the second term of the left hand side of \eqref{D28} can be written as
	$$\begin{array}{l}
	\displaystyle{\int_{Q_T}}a(x,t,u_n,\nabla u_n)(\nabla T_k(u_n)- \nabla w_{\mu}^i)\varphi'(T_k(u_n)-  w_{\mu}^i)h_m(u_n)\,dx\,dt \\
	=\displaystyle{\int_{\{|u_{n}|\leq k\}}a(x,t,u_n,\nabla u_n)(\nabla T_k(u_n)- \nabla w_{\mu}^i)\varphi'(T_k(u_n)-  w_{\mu}^i)h_m(u_n)\,dx\,dt} \\
	+\displaystyle{\int_{\{|u_{n}|> k\}}}a(x,t,u_n,\nabla u_n)(\nabla T_k(u_n)- \nabla w_{\mu}^i)\varphi'(T_k(u_n)-  w_{\mu}^i)h_m(u_n)\,dx\,dt \\
	=\displaystyle{\int_{Q_T}a(x,t,u_n,\nabla u_n)(\nabla T_k(u_n)- \nabla w_{\mu}^i)\varphi'(T_k(u_n)-  w_{\mu}^i)\,dx\,dt }\\
	+\displaystyle{\int_{\{|u_{n}|> k\}}}a(x,t,u_n,\nabla u_n)(\nabla T_k(u_n)- \nabla w_{\mu}^i)\varphi'(T_k(u_n)-  w_{\mu}^i)h_m(u_n)\,dx\,dt ,
	\end{array}$$
	since $m>k$ and $h_m(u_n)=1$ on $\{|u_n|\leq k\}$, we deduce that
	\begin{equation}\label{Dk1}
	\begin{array}{l}
	\displaystyle{\int_{Q_T}}a(x,t,u_n,\nabla u_n)(\nabla T_k(u_n)- \nabla w_{\mu}^i)\varphi'(T_k(u_n)-  w_{\mu}^i)h_m(u_n)\,dx\,dt \\
	=\displaystyle{\int_{Q_T}}\Big(a(x,t,T_k(u_n),\nabla T_k(u_n))-a(x,t,T_k(u_n),\nabla T_k(u))\Big)\\
	\hskip1cm (\nabla T_k(u_n)- \nabla T_k(u))\varphi'(T_k(u_n)-  w_{\mu}^i)\,dx\,dt \\
	+\displaystyle{\int_{Q_T}}a(x,t,T_k(u_n),\nabla T_k(u))(\nabla T_k(u_n)- \nabla T_k(u))\\
	\hskip1cm \varphi'(T_k(u_n)-  w_{\mu}^i)h_m(u_n)\,dx\,dt \\
	+\displaystyle{\int_{Q_T}}a(x,t,T_k(u_n),\nabla T_k(u_n)) \nabla T_k(u)\varphi'(T_k(u_n)-  w_{\mu}^i)h_m(u_n)\,dx\,dt \\
	-\displaystyle{\int_{Q_T}}a(x,t,u_n,\nabla u_n)\nabla w_{\mu}^i\varphi'(T_k(u_n)-  w_{\mu}^i)h_m(u_n)\,dx\,dt \\
	= K_1+K_2+K_3+K_4.
	\end{array}
	\end{equation}
	Using \eqref{D2.3}, \eqref{D3.19} and Lebesgue's theorem we have $a(x,t,T_k(u_n),\nabla T_k(u))$ converges to $a(x,t,T_k(u),\nabla T_k(u))$ strongly in $(L^{p'}(Q_T))^N$ and $\nabla T_k(u_n)$ converges to $\nabla T_k(u)$ weakly in $(L^p(Q_T))^N$, then
	\begin{equation}\label{Dk2}
	K_2 = \varepsilon(n).
	\end{equation}
	Using \eqref{D3.19} and \eqref{D3.25} we have
	\begin{equation}\label{Dk3}
	K_3=\displaystyle{\int_{Q_T}}\overline{a} \nabla T_k(u)\,dx\,dt+ \varepsilon(n,\mu),
	\end{equation}
	For what concerns $K_4$ we can write, since $h_m(u_n)=0$ on $\{|u_n|>m+1\}$
	$$\begin{array}{l}K_4=-\displaystyle{\int_{Q_T}}a(x,t,T_{m+1}(u_n),\nabla T_{m+1}(u_n))\nabla w_{\mu}^i\varphi'(T_k(u_n)-
	w_{\mu}^i)h_m(u_n)\,dx\,dt \\
	=-\displaystyle{\int_{\{|u_n|\leq k\}}}a(x,t,T_{k}(u_n),\nabla T_{k}(u_n)) \nabla w_{\mu}^i\varphi'(T_k(u_n)-
	w_{\mu}^i)h_m(u_n)\,dx\,dt \\
	-\displaystyle{\int_{\{k<|u_n|\leq m+1\}}}a(x,t,T_{m+1}(u_n),\nabla T_{m+1}(u_n)) \nabla w_{\mu}^i\\
	\hskip1cm\varphi'(T_k(u_n)-
	w_{\mu}^i)
	h_m(u_n)\,dx\,dt,\end{array}$$
	and, as above, by letting $n\rightarrow \infty$
	$$\begin{array}{l}K_4 =-\displaystyle{\int_{\{|u|\leq k\}}} \overline{a} \nabla w_{\mu}^i\varphi'(T_k(u)-
	w_{\mu}^i)\,dx\,dt \\
	\quad -\displaystyle{\int_{\{k<|u|\leq m+1\}}}\overline{a} \nabla w_{\mu}^i\varphi'(T_k(u)-
	w_{\mu}^i)h_m(u)\,dx\,dt+ \varepsilon(n),\end{array}$$
	so that, by letting $\mu\rightarrow \infty$
	\begin{equation}\label{Dk4}
	K_4 = -\displaystyle{\int_{Q_T}} \overline{a} \nabla T_k(u)\,dx\,dt+ \varepsilon(n,\mu).
	\end{equation}
	In view of (\ref{Dk1}), (\ref{Dk2}), (\ref{Dk3}) and (\ref{Dk4}), we conclude then that\begin{equation}\begin{array}{l}
	\displaystyle \int_{Q_T}a(x,t,u_n,\nabla u_n)(\nabla T_k(u_n)- \nabla w_{\mu}^i)\varphi'(T_k(u_n)-  w_{\mu}^i)h_m(u_n)\,dx\,dt \\
	=\displaystyle \int_{Q_T} \Big(a(x,t,T_k(u_n),\nabla T_k(u_n))-a(x,t,T_k(u_n),\nabla T_k(u))\Big)\\
	\quad (\nabla T_k(u_n)- \nabla T_k(u))\varphi'(T_k(u_n)-  w_{\mu}^i)\,dx\,dt + \varepsilon(n,\mu).\end{array}\label{D31} \end{equation}
	To deal with the third term of the left hand side of \eqref{D28}, observe that
	$$\begin{array}{l}
	\left|\displaystyle{\int_{Q_T}}a(x,t,u_n,\nabla u_n)\nabla u_n\varphi(\theta_{n}^{\mu,i})h'_m(u_n)\,dx\,dt\right| \\
	\quad\leq\varphi(2k)\displaystyle{\int_{\{m\leq|u_{n}|\leq m +1\}}}a(x,t,u_n,\nabla u_n)\nabla u_n\,dx\,dt.
	\end{array}$$
	Thanks to \eqref{D3.20}, we obtain
	\begin{equation}\begin{array}{l}
	\left|\displaystyle{\int_{Q_T}}a(x,t,u_n,\nabla u_n)\nabla u_n\varphi(\theta_{n}^{\mu,i})h'_m(u_n)\,dx\,dt\right| \leq  \varepsilon(n,m).
	\end{array}\label{D32} \end{equation}
	We now turn to fourth term of the left hand side of (\ref{D28}), we can write
	\begin{equation}
	\begin{array}{l}
	\displaystyle{\left|\int_{\{|u_n|\leq k\}}g_{n}(x,t,u_n,\nabla u_n)\varphi(T_k(u_n)-w^{i}_{\mu})h_m(u_n)dx dt\right|} \\
	\leq  \displaystyle{\int_{\{|u_n|\leq k\}}L_1(k)(L_2(x,t)+|\nabla T_k(u_n)|^p|\varphi(T_k(u_n)-w^{i}_{\mu})h_m(u_n)\,dx\,dt }\\
	\leq  \displaystyle{L_1(k)\int_{Q_T}L_2(x,t)|\varphi(T_k(u_n)-w^{i}_{\mu})|\,dx\,dt }\\
	\quad+  \displaystyle{\frac{L_1(k)}{\alpha}\int_{Q_T}a(x,t,T_k(u_n),\nabla T_k(u_n)) \nabla T_k(u_n)|\varphi(T_k(u_n)-w^{i}_{\mu})|\,dx\,dt },
	\end{array}
	\label{D33}
	\end{equation}
	since $L_2(x,t)$ belong to $L^1(Q_T)$ it is easy to see that
	$$
	\displaystyle{L_1(k)\int_{Q_T}L_2(x,t)|\varphi(T_k(u_n)-w^{i}_{\mu})|\,dx\,dt } =  \varepsilon(n,\mu).
	$$
	On the other hand, the second term of the right hand side of (\ref{D33}), write as
	$$\begin{array}{l}
	\displaystyle{\frac{L_1(k)}{\alpha}\int_{Q_T}a(x,t,T_k(u_n),\nabla T_k(u_n))\nabla T_k(u_n)|\varphi(T_k(u_n)-w^{i}_{\mu})|\,dx\,dt }\\
	=\displaystyle{\frac{L_1(k)}{\alpha}\int_{Q_T}}\Big(a(x,t,T_k(u_n),\nabla T_k(u_n))-a(x,t,T_k(u_n),\nabla T_k(u))\Big)\\
	\hskip1cm (\nabla T_k(u_n)-\nabla T_k(u))|\varphi(T_k(u_n)-w^{i}_{\mu})|\,dx\,dt \\
	+\displaystyle{\frac{L_1(k)}{\alpha}\int_{Q_T}}a(x,t,T_k(u_n),\nabla T_k(u))(\nabla T_k(u_n)-\nabla T_k(u))|\varphi(T_k(u_n)-w^{i}_{\mu})|\,dx\,dt \\
	+\displaystyle{\frac{L_1(k)}{\alpha}\int_{Q_T}}a(x,t,T_k(u_n),\nabla T_k(u))\nabla T_k(u)|\varphi(T_k(u_n)-w^{i}_{\mu})|\,dx\,dt,
	\end{array}
	$$
	and, as above, by letting first $n$ then finally $\mu$ go to infinity, we can easily seen, that each one of last two integrals is of the form $\varepsilon(n,\mu)$. This implies that
	\begin{equation}
	\begin{array}{l}
	\displaystyle{\left|\int_{\{|u_n|\leq k\}}g_{n}(x,t,u_n,\nabla u_n)\varphi(T_k(u_n)-w^{i}_{\mu})h_m(u_n)dx dt\right|} \\
	\leq\displaystyle{\frac{L_1(k)}{\alpha}\int_{Q_T}}\Big(a(x,t,T_k(u_n),\nabla T_k(u_n))-a(x,t,T_k(u_n),\nabla T_k(u))\Big)\\
	\hskip1cm (\nabla T_k(u_n)-\nabla T_k(u))|\varphi(T_k(u_n)-w^{i}_{\mu})|\,dx\,dt +\varepsilon(n,\mu).
	\end{array}
	\label{D34}
	\end{equation}
	Combining (\ref{D28}), (\ref{Dzn}), (\ref{D31}), (\ref{D32}) and (\ref{D34}), we get
	$$
	\begin{array}{l}
	\displaystyle{\int_{Q_T}}\Big(a(x,t,T_k(u_n),\nabla T_k(u_n))-a(x,t,T_k(u_n),\nabla T_k(u))\Big)\\
	\hskip1cm (\nabla T_k( u_n)-\nabla T_k(u))\left(\varphi'(T_k(u)-w^{i}_{\mu})-\frac{L_1(k)}{\alpha}|\varphi(T_k(u_n)-w^{i}_{\mu})|\right) \,dx\,dt \\
	\quad\quad\leq \varepsilon(n,\mu,i,m),
	\end{array}
	$$
	and so, thanks to (\ref{D26}), we have
	\begin{equation*}
	\begin{array}{l}
	\displaystyle{\int_{Q_T}}\Big(a(x,t,T_k(u_n),\nabla T_k(u_n))-a(x,t,T_k(u_n),\nabla T_k(u))\Big)\\
	\quad\quad(\nabla T_k( u_n)-\nabla T_k(u))\,dx\,dt\leq \varepsilon(n).
	\end{array}\label{D35}
	\end{equation*}
	Hence by passing to the limit sup over $n$, we get
	$$
	\begin{array}{l}
	\displaystyle \limsup_{n\rightarrow\infty}
	\displaystyle{\int_{Q_T}}\Big(a(x,t,T_k( u_n),\nabla T_k( u_n))-a(x,T_k( u_n),\nabla T_k( u))\Big)\\
	\quad\quad(\nabla T_k( u_n)-\nabla T_k( u)) \,dx \,dt=0,
	\end{array}
	$$
	This implies that
	\begin{equation}
	T_k(u_n)\to T_k(u) \mbox{ strongly in $L^p(0,T;W^{1,p}_0(\Omega))$  }  \forall k. \label{D3.39}
	\end{equation}
	Now, observe that for every $\sigma>0$,
	\begin{align*}
	&\operatorname{meas}\Big\{(x,t)\in Q_T:\ |\nabla u_n-\nabla u|>\sigma\Big\}\\
	&\leq \operatorname{meas}\Big\{(x,t)\in Q_T:\ |\nabla u_n|>k\Big\}
	+\operatorname{meas}\Big\{(x,t)\in Q_T:\ |u|>k\Big\}\\
	&\quad +\operatorname{meas}\Big\{(x,t)\in Q_T:\
	\big|\nabla T_k(u_n)-\nabla T_k(u)\big|>\sigma\Big\}
	\end{align*}
	then as a consequence of \eqref{D3.39} we have that $\nabla u_n$ converges
	to $\nabla u$ in measure and therefore, always reasoning for a
	subsequence,
	\begin{equation*}
	\nabla u_n\to \nabla u \mbox{ a. e. in }Q_T .\label{D3.40}
	\end{equation*}
	Which implies
	\begin{equation}
	a(x,t,T_k(u_n),\nabla T_k(u_n))\rightharpoonup a(x,t,T_k(u),\nabla T_k(u))
	\mbox{ weakly in } (L^{p'}(Q_T))^N. \label{D3.41}
	\end{equation}
	\subsection*{Step 3 : Equi-integrability of $H_{n}(x,t,\nabla u_n)$ and $g_{n}(x,t,u_n,\nabla u_n)$}
	\quad We shall now prove that $H_{n}(x,t,\nabla u_n)$ converges to $H(x,t,\nabla u)$ and $g_{n}(x,t,u_n,\nabla u_n)$ converges to  $g(x,t,u,\nabla u)$ strongly in $L^{1}(Q_T)$ by using Vitali's theorem. Since $H_{n}(x,t,\nabla u_n)\rightarrow H(x,t,\nabla u)$ a.e. $Q_T$ and $g_{n}(x,t,u_n,\nabla u_n)\rightarrow g(x,t,u,\nabla u)$ a.e. $Q_T$, thanks to \eqref{D2.6} and \eqref{D2.666}, it suffices to prove that $H_{n}(x,t,\nabla u_n)$ and $g_{n}(x,t,u_n,\nabla u_n)$ are uniformly equi-integrable in $Q_T$. We will now prove that $H(x,\nabla u_n)$ is uniformly equi-integrable, we use H\"{o}lder's inequality and \eqref{Dbounded}, we have for any measurable subset $E\subset Q_T$:
	\begin{equation*}\begin{array}{lll}
	\displaystyle{\int_E }|H(x,\nabla u_n)| \,dx\,dt& \leq & \left(\displaystyle{\int_E } h^p(x,t)\,dx\,dt\right)^{\frac{1}{p}}\left(\displaystyle{\int_{Q_T}}|\nabla u_n|^p \,dx\,dt\right)^{\frac{1}{p'}}\\
	\ & \leq & c_1\left(\displaystyle{\int_E } h^p(x,t)\,dx\,dt\right)^{\frac{1}{p}}
	\label{Dequi2}\end{array}\end{equation*}
	which is small uniformly in $n$ when the measure of $E$ is small.
	\par To prove the uniform equi-integrability of $g_{n}(x,t,u_n,\nabla u_n)$. For any measurable subset $E\subset Q_T$ and $m\geq0$,\begin{equation}
	\begin{array}{l}\displaystyle{\int_{E}}|g_{n}(x,t,u_n,\nabla u_n)|\,dx\,dt=  \displaystyle{\int_{E\cap \{|u_n|\leq m\}}}|g_{n}(x,t,u_n,\nabla u_n)|\,dx\,dt\\
	+\displaystyle{\int_{E\cap \{|u_n|> m\}}}|g_{n}(x,t,u_n,\nabla u_n)|\,dx\,dt\\
	\quad\leq \displaystyle{L_1(m)\int_{E\cap \{|u_n|\leq m\}} \left[L_2(x,t)+ |\nabla u_n|^p\right]\,dx\,dt}\\
	+\displaystyle{\int_{E\cap \{|u_n|> m\}}}|g_{n}(x,t,u_n,\nabla u_n)|\,dx\,dt\\
	\quad=  K_1+ K_2.\end{array}\label{Dequi1}\end{equation}
	For fixed $m$, we get
	$$\begin{array}{lll}
	K_1  &\leq & \displaystyle{L_1(m)\int_{E}\left[L_2(x,t)+ |\nabla T_m(u_n)|^{p}\right]\,dx\,dt},
	\end{array}$$
	which is thus small uniformly in $n$ for $m$ fixed when the measure of $E$ is small (recall that $T_m(u_n)$  tends to $T_m(u)$ strongly in $L^p(0,T;W_0^{1,p}(\Omega))$).
	We now discuss the behavior of the second integral of the right hand side of \eqref{Dequi1},  let $\psi_m$ be a function such that
	\begin{eqnarray*}\left\{\begin{array}{lll}
			\psi_m(s) = 0 & \mbox{if} &|s| \leq m-1,\\
			\psi_m(s)= \mbox{sign}(s) & \mbox{if} & |s| \geq m,\\
			\psi'_m(s) = 1 & \mbox{if} &  m - 1 < |s|<  m.\end{array}\right.\end{eqnarray*}
	We choose for $m > 1$, $\psi_m(u_n)$ as a test function  in \eqref{Dpn}, and we obtain
	\begin{align*}
	&\Big[\int_{\Omega} B_m^n(x,u_n) dx\Big]_0^T
	+\int_{Q_T} a(x,t,u_n,\nabla u_n) \nabla u_n\psi'_m(u_n)\,dx\,dt \\
	&+\int_{Q_T}g_{n}(x,t,u_n,\nabla u_n)\psi_m(u_n)\,dx\,dt+\int_{Q_T}H_{n}(x,t,\nabla u_n)\psi_m(u_n)\,dx\,dt\\
	&\quad =\int_{Q_T} f_n\psi_m(u_n)\,dx\,dt,
	\end{align*}where $\displaystyle B_m^n(x,r)=\int_0^r\frac{\partial b_n(x,s)}{\partial s}
	\psi_m(s)ds$, which implies, since $B_m^n(x,r)\geq 0$ and using \eqref{D2.5}, H\"{o}lder's inequality
	$$
	\int_{\{m-1\leq|u_n|\}}|g_{n}(x,t,u_n,\nabla u_n)|\,dx\,dt\leq\int_{E}|H_{n}(x,t,\nabla u_n)|\,dx\,dt + \int_{\{m-1\leq|u_n|\}}|f|\,dx\,dt,
	$$
	and by \eqref{Dbounded}, we have
	$$
	\lim_{m\to \infty}\sup_{n\in
		\mathbb{N}}\int_{\{|u_n|>m-1\}}|g_{n}(x,t,u_n,\nabla u_n)|\,dx\,dt=0.
	$$
	Thus we proved that the second term of the right hand side of \eqref{Dequi1} is also small, uniformly in $n$ and in $E$ when $m$ is sufficiently large. Which shows that $g_{n}(x,t,u_n,\nabla u_n)$ and $H_{n}(x,t,\nabla u_n)$ are uniformly equi-integrable in $Q_T$ as required, we conclude that
	\begin{equation}\begin{array}{lll}
	H_{n}(x,t,\nabla u_n)\to H(x,t,\nabla u)& \mbox{strongly  in }& L^1(Q_T),\\
	g_{n}(x,t,u_n,\nabla u_n)\to g(x,t,u,\nabla u)& \mbox{strongly  in }& L^1(Q_T).
	\label{Dconvhg}\end{array}
	\end{equation}
	\subsection*{Step 4 : }In this step we prove that $u$ satisfies \eqref{D2.8}.
	\begin{lemma} \label{Dlemma58}
		The limit $u$ of the approximate solution $u_n$ of \eqref{Dpn}
		satisfies
		$$
		\lim_{m\to +\infty} \int_{\{m\leq |u|\leq m+1\}}
		a(x,t,u,\nabla u)\nabla u\,dx\,dt=0.
		$$
	\end{lemma}
	\begin{proof}
		Note that for any fixed $m\geq0$, one has
		\begin{align*}
		&\int_{\{m\leq |u_n|\leq m+1\}}a(x,t,u_n,\nabla u_n)\nabla u_n\,dx\,dt\\
		&=\int_{Q_T} a(x,t,u_n,\nabla u_n)(\nabla T_{m+1}(u_n)-\nabla T_{m}(u_n))\,dx\,dt\\
		&=\int_{Q_T} a(x,t,T_{m+1}(u_n),\nabla T_{m+1}(u_n))\nabla T_{m+1}(u_n)\,dx\,dt\\
		&\quad -\int_{Q_T}a(x,t,T_m(u_n),\nabla T_m(u_n))\nabla T_m(u_n)\,dx\,dt.
		\end{align*}
		According to \eqref{D3.41} and \eqref{D3.39}, one can pass to
		the limit as $n\to  +\infty$ for fixed $m\geq 0$, to obtain
		\begin{equation}
		\begin{aligned}
		&\lim_{n\to +\infty} \int_{\{m\leq |u_n|\leq
			m+1\}}a(x,t,u_n,\nabla u_n)\nabla u_n\,dx\,dt\\
		&=\int_{Q_T} a(x,t,T_{m+1}(u),\nabla T_{m+1}(u))\nabla T_{m+1}(u)\,dx\,dt\\
		& - \int_{Q_T} a(x,t,T_m(u),\nabla T_m(u))\nabla T_{m}(u_n)\,dx\,dt\\
		&=\int_{\{m\leq |u_n|\leq m+1\}} a(x,t,u,\nabla u)\nabla u\,dx\,dt.
		\end{aligned} \label{D94}
		\end{equation}
		Taking the limit as $m\to +\infty$ in \eqref{D94} and using the
		estimate \eqref{D3.20} show that $u$ satisfies \eqref{D2.8} and the
		proof is complete.
	\end{proof}
	\subsection*{Step 5 : }In this step we prove that $u$ satisfies \eqref{D2.9} and \eqref{D2.10}. Let $S$ be a function in $W^{2,\infty}(\mathbb{R})$ such that $S'$ has a compact support. Let $M$ be a positive real number such that support of $S'$ is a subset of $[-M,M]$. Pointwise multiplication
	of the approximate equation \eqref{Dpn} by $S'(u_n)$ leads to
	\begin{equation}
	\begin{aligned}
	&\frac{\partial B_S^n(x,u_n)}{\partial t}
	-\operatorname{div}\Big(S'(u_n)a(x,t,u_n,\nabla u_n)\Big)+S''(u_n)a(x,t,u_n,\nabla u_n)\nabla u_n\\
	&\quad\quad\quad\quad+S'(u_n)\Big(g_{n}(x,t,u_n,\nabla u_n)+H_{n}(x,t,\nabla u_n)\Big)=fS'(u_n) \mbox{ in }\mathcal{D}'(Q_T),\\
	&\quad\quad\mbox{ where $\displaystyle B_S^n(x,z)=\int_0^z \frac{\partial
			b_n(x,r)}{\partial r}S'(r)dr$}.
	\end{aligned}\label{D3.44}
	\end{equation}
	In what follows we pass to the limit in \eqref{D3.44} as $n$ tends to $+\infty$.
	
	$\bullet$ Limit of $\displaystyle \frac{\partial B_S^n(x,u_n)}{\partial t}$. Since $S$ is bounded and continuous, $u_n\to u$ a.e. in
	$Q_T$, implies that $B_S^n(x,u_n)$ converges to $B_S(x,u)$ a.e. in $Q_T$ and $L^{\infty}(Q_T)$-$\mbox{weak}^*$. Then $\displaystyle \frac{\partial B_S^n(x,u_n)}{\partial t}$ converges to $\displaystyle\frac{\partial B_S(x,u)}{\partial t}$ in $\mathcal{D}'(Q_T)$ as $n$ tends to $+\infty$.
	
	$\bullet$The limit of $-\operatorname{div}\Big(S'(u_n)a(x,t,u_n,\nabla u_n)\Big)$. Since $\operatorname{supp }(S')\subset [-M,M]$, we have for
	$n\geq M:\ S'(u_n)a_n(x,t,u_n,\nabla u_n)=S'(u_n)a(x,t,T_M(u_n),\nabla T_M(u_n))$ a.e. in $Q_T.$
	\par The pointwise convergence of $u_n$ to $u$, \eqref{D3.41} and the bounded character of $S'$ yield, as
	$n$ tends to $+\infty:\ S'(u_n)a_n(x,t,u_n,\nabla u_n)$ converges to\\ $S'(u)a(x,t,T_M(u),\nabla T_M(u))$ in $(L^{p'}(Q_T))^N,$ and $S'(u)a(x,t,T_M(u),\nabla T_M(u))$ has been denoted by $S'(u)a(x,t,u,\nabla u)$ in equation \eqref{D2.9}.
	
	$\bullet$ The limit of $S''(u_n)a(x,t,u_n,\nabla u_n)\nabla u_n$. Consider the ''energy'' term,\\
	$
	S''(u_n)a(x,t,u_n,\nabla u_n)\nabla u_n=S''(u_n)a(x,t,T_M(u_n),\nabla T_M(u_n))\nabla T_M(u_n)$
	a.e. in $Q_T.$
	\par The pointwise convergence of $S'(u_n)$ to $S'(u)$ and \eqref{D3.41} as $n$ tends to $+\infty$ and the bounded character of $S''$ permit us
	to conclude that $S''(u_n)a_n(x,t,u_n,\nabla u_n)\nabla u_n$ converges to\\ $S''(u)a(x,t,T_M(u),\nabla T_M(u))\nabla T_M(u)$ weakly in $L^1(Q_T).$\\
	Recall that\\
	$S''(u)a(x,t,T_M(u),\nabla T_M(u))\nabla T_M(u)=S''(u)a(x,t,u,\nabla u)\nabla u  \mbox{ a.e. in }Q_T.
	$
	
	$\bullet$ The limit of $S'(u_n)\Big(g_{n}(x,t,u_n,\nabla u_n)+H_{n}(x,t,\nabla u_n)\Big)$. From\\ $\operatorname{supp}(S')\subset [-M,M]$, by
	\eqref{Dconvhg}, we have $S'(u_n)g_{n}(x,t,u_n,\nabla u_n)$ converges to\\ $S'(u)g(x,t,u,\nabla u)$ strongly in $L^1(Q_T)$ and $S'(u_n)H_{n}(x,t,\nabla u_n)$ converge to \\$S'(u)H(x,t,\nabla u)$ strongly in $L^1(Q_T)$, as $n$ tends to $+\infty$.
	
	$\bullet$ The limit of $S'(u_n)f_n$. Since $u_n\rightarrow u$ a.e. in $Q_T$, we have
	$S'(u_n)f_n$ converges to $S'(u)f$ strongly in $L^1(Q_T),$ as $n$ tends to $+\infty.$
	
	As a consequence of the above convergence result, we are in a
	position to pass to the limit as $n$ tends to $+\infty$ in equation
	\eqref{D3.44} and to conclude that $u$ satisfies \eqref{D2.9}.
	
	It remains to show that $B_S(x,u)$ satisfies the initial condition
	\eqref{D2.10}. To this end, firstly remark that, $S$ being
	bounded and in view of (\ref{Db16}), (\ref{Dtk}), we have $B_S^n(x,u_n)$ is bounded in $L^{p}(0,T;W^{1,p}_0(\Omega))$. Secondly,
	\eqref{D3.44} and the above considerations on the behavior of the
	terms of this equation show that $\frac{\partial
		B_S^n(x,u_n)}{\partial t}$ is bounded in
	$L^{1}(Q_T)+L^{p'}(0,T;W^{-1,p'}(\Omega))$. As a consequence (see \cite{Po1}), $B_S^n(x,u_n)(t=0)=B_S^n(x,u_{0n})$
	converges to $B_S(x,u)(t=0)$ strongly in $L^1(\Omega)$. On the other
	hand, the smoothness of $S$ and in view of \eqref{Du0n} imply that
	$B_S(x,u)(t=0)=B_S(x,u_0)$ in $\Omega.$
	As a conclusion, steps 1--5 complete the proof of Theorem
	\ref{Dthm1}.
\end{proof}

\bibliographystyle{unsrt}  


\end{document}